\documentclass[12pt]{amsart}
\usepackage{amsmath,amsfonts,amscd,amsthm,amssymb,latexsym}
\usepackage{xcolor}
\usepackage{amsthm}						
\usepackage{tabu}
\usepackage{enumerate}
\usepackage{adjustbox}
\usepackage{hyperref}

\font\teneufm=eufm10 \font\seveneufm=eufm7 \font\fiveeufm=eufm5
\newfam\frakturfam
\textfont\frakturfam=\teneufm \scriptfont\frakturfam=\seveneufm
\scriptscriptfont\frakturfam=\fiveeufm

\newtheorem{df}{Definition}
\newtheorem{lm}{Lemma}
\newtheorem{theor}{Theorem}
\newtheorem{co}{Corollary}
\newtheorem{rem}{Remark}
\newtheorem{ex}{Example}

\def\bee{\begin{eqnarray}}
	\def\bes{\begin{eqnarray*}}
		\def\eee{\end{eqnarray}}
	\def\ees{\end{eqnarray*}}

\def\Proof{{\sl Proof.}\ }

\title{Just-infinite Jordan  Banach algebras}
\author{V.~N.~Zhelyabin, A.~S.~Mamontov  }

\begin{document}
	
	\maketitle
	
	\newcommand{\Addresses}{{
			\bigskip\noindent
			\footnotesize
			Victor~Zhelyabin, \textsc{Sobolev Institute of Mathematics, Novosibirsk, Russia;}\\\nopagebreak
			\textit{E-mail address: } \texttt{vicnic@math.nsc.ru }
			\\\nopagebreak
			ORCID:\textsc{0000-0001-9371-6363}
			
			\medskip\noindent
			Andrey~Mamontov, \textsc{Sobolev Institute of Mathematics, Novosibirsk, Russia;}\\\nopagebreak
			\textit{E-mail address: } \texttt{andreysmamontov@gmail.com}
			\\\nopagebreak
			ORCID:\textsc{0000-0002-0324-5287}
			\medskip
	}}
	
	\parbox[c]{15cm}
	{\bf Abstract. \it
		By analogy with the well-established notions
		of just-infinite groups and 
		just-infinite  algebras, in particular $C^*$-algebras,
		we initiate a study of
		just-infinite $JB$-algebras, i.e.
		infinite dimensional $JB$-algebras for which
		all proper quotients are finite dimensional.
		We investigate the connections between
		a just-infinite $C^*$-algebra $A$ and
		its Jordan algebra $H(A,*)$ of self-adjoint elements.
		We also show that any just-infinite $JB$-algebra $J$ 
		either is a infinite-dimensional spin factor
		or there exists a $C^*$-algebra $A$ and
		just-infinite norm-closed real $*$-subalgebras 
		$A_1$ and $A_2$ of $A$ such that
		$H(A_1,*)\unlhd J \subseteq H(A_2,*).$
	}
	
	\vspace*{5mm}
	
	\noindent
	{\bf Mathematics Subject Classification (2010):} 
	46L05, 46H99, 46L70, 17C99, 17C65
	
	\medskip
	
	\noindent
	{\bf Key words:} 
	just-infinite group,
	just-infinite $C^*$-algebra, 
	$JB$-algebra, 
	$JC$-algebra, 
	spin factor.
	
	\section{\bf Introduction}
	\hspace*{\parindent}\vspace*{-10pt}
	
	Various finiteness conditions play an important role 
	in the structure theory of any class of rings.
	Naturally, the structural theory for some class of rings starts with the study of finite-dimensional algebras. Next, we move on to study algebras 
	that satisfy various finiteness conditions, 
	such as the minimality condition for ideals. 
	The theory of finite-dimensional algebras, as well as the theory of algebras satisfying the minimality condition for one-sided ideals or quadratic ideals 
	is well developed for both associative algebras 
	and Jordan algebras. 
	
	In \cite{FS} Farkas and Small study so-called 
	nearly finite-dimensional associative algebras, 
	also known as {\it just-infinite dimensional}, or 
	{\it just-infinite} for short.
	These are infinite-dimensional algebras 
	whose every non-zero ideal is of finite codimension. 
	An example of a just-infinite algebra is the ring of polynomials in one variable.
	
	The class of just-infinite algebras
	is sufficiently large. 
	As shown in \cite{PT}, 
	every finitely generated algebra can be
	homomorphically mapped onto a just-infinite algebra. 
	In \cite{FS}, a description of associative
	infinite-dimensional algebras with unity, 
	in which every nonzero one-sided ideal
	has finite codimension, was given.
	Such an algebra is either a division algebra or
	a finite module over its center, 
	which is a finite-dimensional algebra. 
	Just-infinite $\mathbb{N}$-graded associative algebras were studied in \cite{RRZh}. 
	In particular, that paper contains examples
	of finitely generated just-infinite associative algebras that are not PI-algebras, 
	and an example of a just-infinite associative algebra
	that is not Noetherian for one-sided ideals.
	
	It was proved in \cite{FarPen} that just-infinite associative algebras are prime. 
	This article also shows that 
	a just-infinite associative unital PI-algebra 
	is a finitely generated module over its center, 
	and that the center is itself a just-infinite algebra. 
	The research presented in \cite{FarPen} 
	was continued in \cite{BellFarPen}, 
	where it was proven that if 
	$K$ is an algebraically closed field and 
	$A$ is a non-PI right noetherian 
	just-infinite $K$-algebra 
	then $A$ is stably just-infinite. 
	A just-infinite $k$-algebra $A$ is called 
	stably just-infinite 
	if $A \otimes_k K$ is just infinite
	over $K$, for any field extension $K/k$.
	
	In \cite{ZhP}, 
	just-infinite Jordan algebras were studied. 
	It was proven that such algebras
	are prime and non-degenerate.
	It was also proven in \cite{ZhP} 
	that a just-infinite Jordan PI-algebra with unity
	is either a finite module over its just finite center, 
	or is a central order in an algebra of 
	a non-degenerate symmetric bilinear form,
	or is an Albert ring. 
	In \cite{ZhPan}, 
	the connections between just-infinite 
	associative superalgebras and
	adjoint Jordan superalgebras were studied. 
	It was shown that if an associative superalgebra is 
	just-infinite, then its adjoint Jordan superalgebra is also just-infinite.
	
	It has already been noted that just-infinite associative Jordan algebras are prime. 
	For Lie algebras, this result is false.  
	In \cite{ShalZel}, 
	an example of a just-infinite Lie algebra
	that is not semi-prime, was constructed; 
	moreover, it has an abelian ideal of codimension one. Just-infinite $\mathbb{N}$-graded Lie algebras
	were studied in \cite{GMS}.
	
	There is a significant interest in so-called 
	just-infinite groups. 
	A group is said to be just-infinite if it is infinite
	and all its proper quotients are finite. 
	J.~Wilson's paper \cite{Wilson} is a pioneering systematic study in this area.
	In \cite{GupSid} N.~Gupta and S.~Sidki constructed, for each prime $p$, a 2-generator infinite $p$-group $E$ such that every finite $p$-group is
	embeddable in $E$ and, for $p$ odd, every proper quotient of $E$ is finite, i.e. $E$ is just-infinite.
	In \cite{GrigShum} R.~Grigorchuk and P.~Shumyatsky provide recent survey and new results on just infinite groups, and their relations to $C^*$-algebras.

	The work \cite{GMR} 
	of R.~Grigorchuk, M.~Musat, and M.~Rørdam 
	classifies just-infinite $C^*$-algebras
	in terms of their primitive ideal space. 
	In particular, this work discusses
	when $C^*$-algebras and 
	$^*$-algebras associated with a discrete group
	are just-infinite.
	
	The purpose of our research is to study Jordan-Banach algebras, also known as $JB$-algebras.  
	$JB$-algebras were introduced in \cite{AlShSt}. 
	Some important properties of these algebras
	were also studied there. 
	Examples of $JB$-algebras are Jordan algebras
	of self-adjoint elements of $C^*$-algebras. 
	Jordan subalgebras of self-adjoint operators
	on a complex Hilbert space that are norm-closed 
	are called $JC$-algebras. 
	These algebras are Jordan analogs
	of concrete $C^*$-algebras. 
	The structure of $JC$-algebras is quite well understood, and is close to that of $C^*$-algebras 
	(see \cite{EffSt,Stor,  Stor-I,Stor-II, Top, Top-I}).
	
	In this work, we prove that a Jordan algebra 
	of self-adjoint elements of a $C^*$-algebra is just-infinite if and only if 
	the underlying $C^*$-algebra is just-infinite.
	We also show that any just-infinite 
	$JB$-algebra is $JC$. 
	More precisely, this algebra is either an infinite-dimensional spin factor or a subalgebra of self-adjoint elements of a real associative just-infinite algebra. It also contains an ideal that is a just-infinite Jordan algebra of self-adjoint elements of a real algebra.
	
	\section{\bf Jordan algebras of self-adjoint elements of just-infinite $C^*$-algebras}
	\hspace*{\parindent}\vspace*{-10pt}
	
	Let $\mathbb{K}$ be
	the field of real numbers $\mathbb{R}$ 
	or the field of complex numbers $\mathbb{C}$.
	
	Let $A$ be an algebra over $\mathbb{K}$. 
	An involution on the algebra $A$ is 
	a conjugate-linear anti-automorphism of order two, usually denoted $x\mapsto x^*$. 
	In other words, 
	$$(x + y)^* = x^* + y^*,
	(xy)^* = y^*x^*, (\lambda x)^* = \overline{\lambda}x^*,(x^*)^* = x,$$ 
	for all $x, y \in  A$, $\lambda \in \mathbb{K}$. 
	An algebra over $\mathbb{K}$ equipped with an involution is called an involutive algebra ($ ^*$-algebra).
	
	A normed algebra is an algebra over $\mathbb{K}$
	being a normed space with a norm $\Vert\cdot \Vert$, 
	which is submultiplicative 
	($\Vert xy\Vert\leq \Vert x \Vert \Vert y\Vert$ 
	for all $x, y \in  A$).
	
	An involutive normed algebra $A$ 
	is a normed algebra equipped with an involution 
	$x\mapsto x^*$  
	such that 
	$\Vert x^*\Vert=\Vert x\Vert$ 
	for all $x \in A$.  
	Moreover, if additionally $A$ is a Banach algebra, 
	it is called an involutive Banach algebra.
	
	A $C^*$-algebra is an associative  involutive Banach algebra $A$ over $\mathbb{C}$ satisfying the $C^*$-axiom: 
	$$\Vert x^*x \Vert = \Vert x \Vert^2  \text{ for all } x \in A.$$
	
	Let $A$ be a $C^*$-algebra. 
	Note that $A$ may be considered
	as an algebra over $\mathbb{R}$. 
	Any $\mathbb{R}$-subalgebra of the $C^*$-algebra 
	will be called a real subalgebra.
	An element $a$ of $A$ is called self-adjoint if $a^*=a$. Let $H(A,*)$ be the set of self-adjoint elements of $A$.
	
	Note that an involutive normed algebra $A$,
	whose norm satisfies the C*-axiom, 
	does not contain one-sided nil-ideals. 
	Indeed, let $I$ be a left nil-ideal of $A$. 
	Then $x^*x$ is a nilpotent element for $x\in I$. 
	It is clear that $A$ does not contain any nonzero 
	self-adjoint nilpotent element. 
	Therefore, $x^*x=0$ and $x=0$.
	
	\begin{lm} \label{lm4} 
		Let $A$ be a $C^*$-algebra and
		let $I$ be a closed ideal of the $\mathbb{R}$-algebra $A$.
		Then $I$ is an ideal of the $C^*$-algebra $A$.
	\end{lm}
	
	\Proof 
	Let $I$ be a closed ideal of the $\mathbb{R}-$algebra $A$.
	Then $I\cap iI$ is an ideal of the $C^*$-algebra $A$. 
	This ideal is closed.
	
	We claim that $I$ is an ideal of the $C^*$-algebra $A$, that is, $iI\subseteq I$. 
	Consider the vector space $P=I+iI$ in the algebra $A$. 
	It is clear that $P$ is an ideal
	of the $C^*$-algebra $A$. We have
	$$P^2\subseteq P(I+iI)\subseteq PI+(iP)I\subseteq I.$$ Since $P^2=(iP)(iP)\subseteq iP^2$, then $P^2\subseteq iP^2\subseteq iI.$ 
	Consequently, $P^2\subseteq I\cap iI$.
	
	Since $I\cap iI$ is a closed ideal
	of the $C^*-$algebra $A$, 
	then $(I\cap iI)^*\subseteq I\cap iI$  
	(see \cite{Dixmier}[Proposition~1.8.2]). 
	Therefore, 
	the quotient $A/(I\cap iI)$ is a $C^*$-algebra. 
	From the inclusion $P^2\in I\cap iI$, 
	we can conclude that the image of $P$ 
	in the quotient $C^*$-algebra $A/(I\cap iI)$ 
	is nilpotent. 
	Therefore, $P\subseteq I\cap iI$. 
	Consequently, since $P \subseteq I$, 
	we have $P = I = iI$. 
	This means that $I$ is an ideal
	of the $C^*$-algebra $A$.
	\qed

	\begin{lm} \label{lm5} 
		A $C^*$-algebra $A$ is finite-dimensional
		if and only if
		$A$ is a finite-dimensional $\mathbb{R}-$algebra.
	\end{lm}
	
	\Proof 
	Let $A$ be a finite-dimensional $C^*$-algebra. 
	Let $e_1,\ldots, e_n$ be a basis
	of the $C^*$-algebra $A$. 
	It is well known that 
	$$e_n=f_n+ig_n,\, f_n,g_n\in H(A,*).$$ 
	Then, for any $a\in A$, we have 
	$$a=\sum_{k=1}^n\alpha_k e_k,$$ where
	$\alpha_1,\ldots,\alpha_n\in \mathbb{C}.$
	
	Let $\alpha_k=\beta_k+i\gamma_k,$ where $\beta_k,\gamma_k\in \mathbb{R}$ and
	$k=1,\ldots , n$. 
	Then
	$$a=\sum_{k=1}^n(\beta_kf_k-\gamma_kg_k)+(\beta_k ig_k+\gamma_kif_k).$$ 
	That implies that the algebra $A$ as a vector space 
	over the field of real numbers $\mathbb{R}$ 
	is generated by the elements $f_1,g_1,if_1,ig_1,\ldots,f_n,g_n,if_n,ig_n$. Consequently, 
	$A$ is finite-dimensional over $\mathbb{R}.$
	
	It is clear that, 
	if $A$ is finite-dimensional over real numbers, 
	it must also be finite-dimensional over complex numbers.
	\qed
	
	\begin{df} 
		A Banach algebra $A$ over the field $\mathbb{K}$
		is said to be just-infinite 
		if it is infinite dimensional, 
		and for each non-zero closed two-sided ideal $I$ in $A$, the quotient $A/I$ is finite dimensional. 
		In particular, the $C^*$-algebra $A$ is said to be 
		just-infinite if it is infinite dimensional, 
		and for each non-zero closed two-sided ideal $I$ in $A$, the quotient $A/I$ is finite dimensional.
	\end{df}
	
	\begin{co} \label{co1} 
		A $C^*$-algebra is just-infinite 
		if and only if
		it is just-infinite over the real numbers.\end{co}
	
	\Proof 
	Let $A$ be a just-infinite $C^*$-algebra. 
	We claim that $A$ is 
	a just-infinite algebra over the real numbers. 
	By Lemma \ref{lm5}, the algebra 
	$A$ is infinite-dimensional over $\mathbb{R}$. 
	Let $I$ be a non-zero closed ideal of the
	$\mathbb{R}$-algebra $A$. 
	By Lemma \ref{lm4}, 
	$I$ is a closed ideal of the $C^*$-algebra $A$. 
	Then the quotient algebra $A/I$ 
	is a finite-dimensional $C^*$-algebra. 
	Therefore, by Lemma \ref{lm5},
	$A/I$ is a finite-dimensional $\mathbb{R}$-algebra. Consequently,
	$A$ is a just-infinite algebra over the real numbers.
	
	Let $A$ be a just-infinite algebra over the real numbers. By Lemma \ref{lm5}, 
	the $C^*$-algebra $A$ is infinite-dimensional. 
	Let $I$ be a non-zero closed ideal 
	of the $C^*$-algebra $A$. 
	Then $I$ is a closed ideal of the 
	$\mathbb{R}$-algebra $A$. 
	Therefore, the quotient algebra $A/I$ 
	is finite-dimensional over $\mathbb{R}$. 
	By Lemma \ref{lm5}, 
	$A/I$ is a finite-dimensional $C^*$-algebra.
	Consequently, the $C^*$-algebra $A$ is just-infinite. \qed
	
	Let $A$ be a self-adjoint real subalgebra 
	of a $C^*$-algebra, that is, 
	$A^*\subseteq A$ and $A$ is an algebra over $\mathbb{R}$. 
	Let $J=H(A,*)$ be the set of self-adjoint elements of $A$.
	Define the multiplication $(\circ)$ on $J=H(A,* )$ 
	by the rule 
	$$x\circ y=\frac{1}{2}(xy+yx),\, x,y\in J.$$  
	Assume that $A$ is norm-closed 
	with respect to the $C^*$-norm of the $C^*$-algebra.
	Then $J$ is a $JB$-algebra (see \cite{AlShSt, Han-OlSt}), 
	that is, the Jordan algebra $J$ is 
	an algebra over the field of real numbers 
	equipped with a complete norm 
	satisfying the following requirements for $a,b\in J$:
	$$\Vert a\circ b \Vert\leq \Vert a \Vert \Vert b\Vert,$$
	$$\Vert a^2\Vert=\Vert a\Vert^2,$$
	$$\Vert a^2\Vert\leq \Vert a^2+b^2\Vert.$$
	
	\begin{ex} 
		Let $\mathcal{H}$ be a complex Hilbert space and 
		let $B(\mathcal{H})$ be the algebra of 
		bounded operators of $\mathcal{H}$. 
		Then $B(\mathcal{H})$ is a $C^*$-algebra 
		with the usual operator norm
		$$\Vert T\Vert=\sup_{\Vert x\Vert=1}\Vert Tx\Vert$$ 
		and an involution on $B(\mathcal{H})$, 
		which is the map $^*:T \mapsto T^*$ 
		sending an operator to its adjoint operator.
		Consider the vector space $B(\mathcal{H})_{sa}$ 
		of all self-adjoint operators in the algebra $B(\mathcal{H})$, that is,  $B(\mathcal{H})_{sa}=H(B(\mathcal{H}),*) $.
		Then $B(\mathcal{H})_{sa}$ is a $JB$-algebra. 
	\end{ex}
	
	In the following lemmas, we assume that 
	$A$ is a real subalgebra of a $C^*$-algebra, 
	which is self-adjoint and norm-closed 
	with respect to the $C^*$-norm. 
	Moreover, $J=H(A,*)$. 
	
	We note that $H(P,*)$ is nonzero 
	for any nonzero ideal $P$ of $A$. 
	Indeed, if $P$ is a nonzero ideal of  the algebra $A$, then $p^* p \neq 0$ for any nonzero element $p \in P$. Therefore, $H(P, *) \neq 0$,  because $p^*p\in P$.
	
	\begin{lm} \label{lm1}  
		Let $V$ be a norm-closed subspace of the algebra $J$ and let $B$ be a self-adjoint subspace of the algebra $A$ such that
		$B\cap J\subseteq V$. 
		Denote by $\overline{B}$ the closure of the
		subspace $B$ with respect to the norm of $A$. 
		Then $\overline{B}\cap J\subseteq V$.
	\end{lm}
	
	\Proof 
	Let $a\in \overline{B}\cap J$. 
	Since $a\in \overline{B}$, there exists a sequence
	$\{a_n\}$ of the elements of $B$ such that
	$a_n\rightarrow a$.
	Since the involution $*$ is a continuous map, we have $a_n^*\rightarrow a^*=a$.
	Therefore, $\frac{1}{2}(a_n+a_n^*)\rightarrow a$.
	Since $B$ is a self-adjoint ideal and $a_n\in B$, 
	then $a_n^*\in B$.
	Hence, $a_n+a_n^*\in B\cap J\subset V$. 
	Since $V$ is a closed subspace, we have $a\in V$.
	Therefore, $\overline{B}\cap J\subseteq V$.\qed
	
	\begin{lm} \label{lm2}  
		Let $\{B_\lambda\}_{\lambda\in \Lambda}$ 
		be a family of self-adjoint subspaces of $A$ 
		and let $V$ be a subspace of $J$ such that
		$B_\lambda \cap  J\subseteq V$ for all $\lambda\in\Lambda$. 
		Put $B=\sum_{\lambda}B_{\lambda}$. 
		Then $B\cap J\subseteq V$.\end{lm}
	
	\Proof  
	Let $a\in B\cap J$. 
	Then $a=\sum_{i=1}^{n}  a_{\lambda_i}$,
	where $a_{\lambda_i}\in B_{\lambda_i}.$ 
	Since $B_{\lambda_i}$ are self-adjoint, 
	we have $a_{\lambda_i}^*\in B_{\lambda_i}.$ 
	Hence,
	$$a=\frac{1}{2}(a+a^*)=\frac{1}{2}\sum_{i=1}^n(
	a_{\lambda_i}+a_{\lambda_i}^*)\in\sum_{i=1}^n B_{\lambda_i}\cap J\subseteq V, $$ since $a^*=a$. Consequently, $B\cap J\subseteq V$.
	\qed
	
	Denote by $A^\#$ the algebra with a joined unit.
	
	\begin{lm} \label{lm3} 
		Let $I$ be a closed ideal of $J$. 
		Then the $\mathbb{R}$-algebra $A$ 
		contains a self-adjoint normed-closed ideal $\overline{P}$ 
		such that $I=H(\overline{P},*)$. 
		In particular, 
		if $A$ is a $C^*$-algebra and $J=H(A,*)$ 
		then $A$ contains a closed ideal $\overline{P}$ 
		such that $I=H(\overline{P},*)$. \end{lm}
	
	\Proof
	Let $0\neq a\in I$. 
	Then $a^3\neq 0$ since $a$ is self-adjoint. 
	Put $P_a=A^\#a^3A^\#$.
	It is well known that $P_a\cap H(A,*)\subseteq I$ 
	(see \cite{KMcC,ZHSlShShir}). 
	It is clear that $P_a$ is a self-adjoint ideal of 
	the $\mathbb{R}$-algebra $A$. 
	
	Let $\overline{P_a}$ be the closure of the ideal
	$P_a$ with respect to the $C^*$-norm of the $\mathbb{R}$-algebra $A$. By Lemma \ref{lm1}, $\overline{P_a} \cap J\subseteq I$. Let $P=\sum_{a\in I}\overline{P_a}$.
	Then $P\cap J\subseteq I$ by Lemma \ref{lm2}.
	Let $\overline{P}$ be the closure of the ideal $P$ 
	in the algebra $A$.
	By Lemma \ref{lm1}, $\overline{P}\cap J\subseteq I$.
	
	Now we show that $I\subseteq P\subseteq \overline{P}$. 
	Note that $J\cap\overline{P_a}$ 
	is a norm-closed ideal of the $JB$-algebra $J$.  
	Consider the quotient algebra
	$J/J\cap \overline{P_a}$. 
	It is a $JB$-algebra by \cite{AlShSt} [Lemma 9.2]. 
	In the algebra
	$J/J\cap \overline{P_a}$ we have $\widetilde{a}^3=0$, 
	where $\widetilde{a}$ is the image
	of an element $a$ in the quotient algebra
	$J/J\cap \overline{P_a}$.
	Therefore, $\widetilde{a}=0$ since
	$\Vert\widetilde{a}^2\Vert=\Vert \widetilde{a} \Vert ^2$.
	Consequently, $a\in \overline{P_a}$.
	Since 
	$a\in \overline{P_a}\subseteq P\subseteq \overline{P}$ 
	for all $a\in I$ we have 
	$I\subseteq P\subseteq \overline{P}$.
	
	Thus, $I=H(\overline{P},* )$.
	
	Let $A$ be a $C^*$-algebra. 
	Then $A$ is an algebra over the real field $\mathbb{R}$.
	As we have seen above, $I=H(\overline{P},* )$. 
	By Lemma \ref{lm4}, $\overline{P}$ is a norm-closed ideal of the $C^*$-algebra $A$.  
	\qed
	
	Let $A$ be a normed algebra over $\mathbb{K}$. 
	We say that the algebra $A$ is {\it topologically prime} 
	if for any two nonzero closed ideals $P$ and $Q$ in $A$, their product $PQ$ is nonzero. 
	An arbitrary algebra is called prime if, 
	for any two nonzero ideals $P$ and $Q$ in the algebra, their product $PQ$ is nonzero. 
	Recall, that an ideal is called {\it trivial} 
	if $ab=0$ for all $a,b\in I$.
	
	\begin{lm} \label{lm6} 
		Let $A$ be a normed algebra over the field $\mathbb{K}$. 
		Then $A$ is prime 
		if and only if 
		$A$ is topologically prime. 
		In particular, a $C^*$-algebra ($JB$-algebra) is prime 
		if and only if 
		it is topologically prime. 
		Let $A$ be an algebra containing no trivial ideals. 
		Then for any two  ideals $P$ and $Q$ of $A$ 
		the equation $PQ=0$ implies that $P\cap Q=QP=0$.   \end{lm}
	
	\Proof 
	It is easy to see that if $A$ is a prime algebra
	then $A$ is a topologically prime algebra.
	
	Let $P$ and $Q$ be non-zero ideals of $A$. 
	Assume that $A$ is topologically prime and $PQ=0$.  
	Let $\overline{P}$ be the closure of the ideal $P$. 
	It is clear that $\overline{P}Q=0$. 
	Similarly,  $\overline{P}\cdot\overline{Q}=0$. 
	Therefore, we get
	$\overline{P}=0$ or $\overline{Q}=0$. 
	Consequently, $P=0$ or $Q=0$, i.e $A$ is a prime algebra.
	
	Assume that $A$ contains no trivial ideals. 
	Let $PQ=0$. 
	Then the ideal $P\cap Q$ is trivial 
	because $(P\cap Q)^2\subseteq PQ$.
	Therefore, $P\cap Q=0$. 
	Since $QP\subseteq P\cap Q$, we get $QP=0$.\qed
	
	Let $A$ be an arbitrary associative algebra. 
	Denote by $[x,y]=xy-yx$ 
	the commutator of the elements $x,y$. 
	Then for any elements $a,b,c\in A$, 
	the following identities hold:
	
	\bee \label{Qcomm}[a,b]^2=\frac{1}{2}[a,[a,b^2]]-b\circ [a,[a,b]];\eee
	
	\bee \label{Tcomm}  [a\circ b,c]+[b\circ c,a]+[c\circ a,b]=0. \eee
	
	In each Jordan algebra $J$ the following identity holds:
	
	\bee
	\label{Jord-Id}
	[(xy)z]t+[(yt)z]x+[(tx)z]y=(xy)(zt)+(xt)(yz)+(xz)(yt).\eee
	
	Denote by $(x,y,z)=(xy)z-x(yz)$ the associator of 
	the elements $x,y,z$ in $J$. 
	Recall that the center of a Jordan algebra $J$ is 
	$$Z(J)=\{z\in J \mid (z,a,b)= 0 \text{ for all } a,b\in J  \}.$$
	
	The involution $j$ of an associative algebra over arbitrary field is an anti-isomorphism of order two 
	that fixes elements of the ground field. 
	Meanwhile, an involution conjugates complex scalars in  an involutive algebra.
	
	\begin{lm} \label{lm6-II} 
		Let $A$ be an associative algebra over the field $\mathbb{R}$ 
		with an involution $j$ such that the Jordan algebra
		$J=H(A,j)$ is a Jordan normed algebra whose norm satisfies $\Vert a^2\Vert =\Vert a\Vert^2 $ for all $a\in J$. 
		Assume that the identity $1$ of $J$ 
		is a sum of central orthogonal idempotents $e_1,\ldots,e_n$. Then 
		$1, e_1,\ldots,e_n$ are in the center of the algebra $A$. \end{lm}
	
	\Proof 
	Let $h\in J$. 
	Then $h=h_1+\ldots +h_n$, where $h_k=h\circ e_k$. 
	We have
	$h_k\circ e_k=h_k$ and $h_k\circ e_l=0$ for $k\neq l$.
	By (\ref{Tcomm}), 
	$$[h_k,e_l]=[h_k\circ e_k,e_l]=-[e_k\circ e_l,h_k]-[e_l\circ h_k,e_k].$$ 
	From here, we obtain that $[h_k,e_l]=0$ 
	for all $k,l=1,\ldots,n$.
	Therefore, $[h,e_k]=0$ 
	for all $h\in J$ and $k=1,\ldots,n$.
	
	Let $s\in A$ such that $s^j=-s$. 
	Then $$[e_k,s]^j=[s^j,e_k^j]=-[s,e_k]=[e_k,s].$$ Consequently, $[e_k,s]\in J$. 
	Then $[e_l,[e_k,s]]=0$ for all $k,l=1,\ldots,n$. 
	By (\ref{Qcomm}), 
	$$[e_k,s]^2=\frac{1}{2}[e_k,[e_k,s^2]]-s\circ [e_k,[e_k,s]].$$ 
	Note that $s^2\in J$. 
	Therefore, $[e_k,s^2]=0$. 
	Then $[e_k,[e_k,s^2]]=0$. 
	Consequently, $[e_k,s]^2=0$. 
	Since $[e_k,s]\in J$ and the equality $\Vert a^2\Vert =\Vert a\Vert^2$ holds, then $[e_k,s]=0$.
	
	For any $a\in A$ we have $a=h+s$, where 
	$h^j=h$, $s^j=-s$. 
	Therefore, $[e_k,a]=0$ for all $k=1,\ldots,n$.
	
	Thus, $1, e_1,\ldots,e_n$ are in the center 
	of the algebra $A$.
	\qed
	
	\begin{rem} \label{rm1}
		Let $A$ be a real associative involutive normed algebra 
		whose norm satisfies the $C^*$-axiom. 
		Assume that $1$ is the identity of $H(A,*)$. 
		Then $1$ is the identity of $A$.
	\end{rem}
	Indeed, let $a \in A$, and let $x = a - a \cdot 1$. 
	Then $x^*x\in J$. 
	Therefore, $1\circ x^*x=x^*x$. 
	On the other hand, by Lemma \ref{lm6-II},  
	$1\circ x^*x=0$.
	Consequently, $x^*x=0$. 
	Having $\Vert x\Vert^2=\Vert x^*x\Vert=0$, we get $x=0$.  
	Thus, $1$ is the identity $A$.

	A family $\{u_\alpha\}_{\alpha\in \Lambda }$ of elements 
	in a $JB$-algebra $J$ is said to be 
	an {\it increasing approximate identity} for $J$ 
	indexed by a directed set $\Lambda$
	(see  \cite{AlShSt, Han-OlSt}) if
	
	(i) $\Vert u_\alpha \Vert\leq 1$ for all $\alpha$,
	
	(ii) $0\leq u_\alpha\leq u_\beta$ whenever $\alpha\leq \beta$,
	
	(iii) $\lim_\alpha \Vert a-a\circ u_\alpha\Vert=0$ for all $a\in J$. 
	
	It was proven in \cite{AlShSt,Han-OlSt} 
	that a norm closed ideal of $J$ 
	has an increasing approximate identity.
	
	The following Lemma holds.
	
	\begin{lm} \label{lm6-III} 
		Let $A$ be a real subalgebra of a $C^*$-algebra 
		which is self-adjoint and norm closed 
		with respect to the $C^*$-norm. 
		Let $P$ be a self-adjoint and norm closed ideal of $A$ 
		and let $\{u_\alpha\}_{\alpha\in \Lambda }$ 
		be an increasing approximate identity for $H(P,*)$. 
		Then $\{u_\alpha\}_{\alpha\in \Lambda }$ is 
		an increasing approximate identity for $P$,  that is, 
		$$\lim_\alpha \Vert a-a u_\alpha\Vert=0$$ for all $a\in P$.
	\end{lm}
	
	\Proof 
	We can assume that $J$ has the unity $1$. 
	Then, by Remark  \ref{rm1}, 
	this unity $1$ is the identity of the algebra $A$. 
	
	Let $\{u_\alpha\}_{\alpha\in \Lambda }$ 
	be an increasing approximate identity for $H(P,*)$. 
	First we show that
	$$\lim _\alpha \Vert (1-u_\alpha)h^2(1-u_\alpha)\Vert= 0$$ 
	for all $h\in H(P, *)$. 
	Indeed,
	$$(1-u_\alpha)h^2(1-u_\alpha)=2(h^2\circ (1-u_\alpha))\circ
	(1-u_\alpha)-h^2\circ (1-u_\alpha)^2.$$ 
	We have
	$$\Vert(h^2\circ (1-u_\alpha))\circ (1-u_\alpha)\Vert\leq 2\Vert(h^2\circ (1-u_\alpha)\Vert$$ 
	because $\Vert 1-u_\alpha\Vert\leq 2$. 
	Therefore,
	$\lim_\alpha \Vert(h^2\circ (1-u_\alpha))\circ (1-u_\alpha)\Vert=0.$ 
	In the Jordan algebra $H(A,*)$, using (\ref{Jord-Id}) 
	with $x=y=h$ and $z=t=(1-u_\alpha)$, 
	we obtain
	$$h^2\circ (1-u_\alpha)^2=-2(h\circ (1-u_\alpha))^2+
	2[(h\circ (1-u_\alpha))\circ (1-u_\alpha)]\circ h+[h^2\circ (1-u_\alpha)]\circ (1-u_\alpha).$$ 
	Therefore, $\lim_\alpha \Vert h^2\circ (1-u_\alpha))^2\Vert=0.$
	Consequently, 
	$\lim _\alpha \Vert(1-u_\alpha)h^2(1-u_\alpha)\Vert= 0$. 
	Then 
	$$\lim _\alpha \Vert(1-u_\alpha)h(1-u_\alpha)\Vert= 0$$ for any positive element $h\in H(P,*)$.
	
	For every element $h\in H(P,*)$, we have $h=h_+-h_-$, where $h_+$ and $h_-$ are positive elements of $H(P,*)$. Therefore, 
	$$\lim_\alpha \Vert(1-u_\alpha)a^*a(1-u_\alpha)\Vert= 0$$ 
	for all $a\in P$. 
	From here, we get
	$$0=\lim _\alpha \Vert(1-u_\alpha)a^*a(1-u_\alpha)\Vert=\lim _\alpha \Vert (a(1-u_\alpha))^*a(1-u_\alpha)\Vert=\lim _\alpha \Vert a(1-u_\alpha)\Vert^2. $$
	Thus $$\lim _\alpha \Vert a-au_\alpha\Vert=0$$ 
	for all $a\in P$. 
	\qed
	
	The following corollary is proven in the standard way
	
	\begin{co} \label{co2}  
		Let $A$ be a real subalgebra of a $C^*$-algebra 
		which is self-adjoint and norm closed
		with respect to the $C^*$-norm. 
		Let $P$ be a self-adjoint and norm closed ideal of $A$. Then the quotient algebra $A/P$ is a Banach $*$-algebra whose norm satisfies the $C^*$-axiom. 
	\end{co}
	
	\begin{theor} \label{th1} 
		Let $A$ be a real subalgebra of a $C^*$-algebra 
		which is self-adjoint and norm-closed 
		with respect to the $C^*$-norm. 
		Then the algebra $A$ is just-infinite 
		if and only if 
		$H(A, *)$ is just-infinite.\end{theor}
	
	\Proof 
	Let $I$ be a closed non-zero ideal of $H(A,*)$. 
	By Lemma \ref{lm3}, $I=H(P,*)$, 
	where $P$ is a self-adjoint norm-closed ideal of the algebra $A$.
	
	We claim that $H(A,*)/I\cong H(A/P,*)$ 
	is an isomorphism of algebras.
	Indeed, the mapping $\phi: a+I\mapsto a+P$ 
	is an embedding of the Jordan algebra $H(A,*)/I$ 
	in the Jordan algebra $H(A/P,*)$. 
	Let $a+P\in H(A/P,*)$. 
	Then $a-a^*\in P$. 
	We have $$a=\frac{a+a^*}{2}+\frac{a-a^*}{2}.$$ 
	Therefore, in the quotient algebra $A/P$, we have
	$a + P = \frac{a+a^*}{2} + P$. 
	This implies that $\phi$ is an isomorphism.
	
	Assume that $A$ is a just-infinite algebra. 
	Recall, that $A$ does not contain one-sided nil-ideals,
	in particular, it has no trivial ideals.
	We claim that $A$ is a prime algebra.
	Indeed, if $Q_1,Q_2$ are closed ideals of $A$ such that $Q_1Q_2=0$, then, by 
	Lemma \ref{lm6}, $Q_1\cap Q_2=0$. 
	Therefore the natural homomorphism 
	$\pi: A \mapsto A/Q_1 \oplus A/Q_2$ is injective.
	Since $A$ is infinite dimensional, one of the summands should also be 
	infinite dimensional.
	Therefore either $Q_1=0$ or $Q_2=0$, 
	as claimed.

	We further claim that 
	$H(A,*)$ is infinite dimensional over $\mathbb{R}$.
	Assume the contrary. 
	Being a $JB$-algebra, $H(A,*)$ is nil-semisimple. 
	Since $H(A,*)$ is a finite-dimensional algebra, 
	it is a unital algebra and its identity
	$$1 = e_1 + \ldots + e_n$$ 
	is a sum of central orthogonal idempotents 
	$e_1, \ldots, e_n$. 
	Moreover, 
	$$H(A,*)=J_1\oplus\ldots\oplus J_n,$$ 
	where $J_1,\dots,J_n$ are simple subalgebras of $H(A,*)$. 
	By Lemma \ref{lm6-II}, $e_1,\ldots e_n\in Z(A)$, 
	where $Z(A)$ is the center of $A$.
	
	Since $A$ is prime, then $1$ is the identity of $A$, otherwise $Q=\{a-a\cdot 1 \mid a\in A\}$ is an ideal of $A$ and $Q\cdot 1=1\cdot Q=0$. 
	Moreover, $e_kA$, $e_lA$ are ideals of $A$ 
	and $e_kA\cdot e_lA=0$. 
	Consequently, $n=1$ and $H(A,*)$ is a simple algebra.  Let $Q$ be an ideal of $A.$ 
	Then $H(Q^*Q,*)$ is an ideal of $H(A,*).$ 
	Therefore, either $H(Q^*Q,*)=H(A,*)$ or $H(Q^*Q,*)=0$.  Consequently, either $Q=A$ or $Q=0$,  
	that is, $A$ is simple. 
	
	The algebra $H(A,*)$ is finite dimensional. 
	Hence, it is a PI-algebra. 
	Therefore, by Theorem~6 of \cite{Amitsur}, 
	$A$ is PI-algebra. 
	Consequently, $A$ is finite-dimensional over the center $Z(A)$. 
	It is easy to see that the center of $A$ is
	either $\mathbb{R}$ or $\mathbb{C}$. 
	Therefore, $A$ is finite-dimensional, a contradiction.
	Thus, $H(A,*)$ is an infinite dimensional algebra over $\mathbb{R}$.
	
	Let $I$ be a closed non-zero ideal of $H(A,*)$. 
	As we have seen above $H(A,*)/I\cong H(A/P,*)$. 
	Since $ H(A/P,*)$ is finite-dimensional then
	$H(A,*)/I$ is finite-dimensional. 
	Therefore, $H(A,*)$ is a just-infinite Jordan algebra.
	
	Conversely, 
	let $H(A,*)$ be a just-infinite Jordan algebra. 
	Then $A$ is infinite dimensional. 
	
	Let $P$ be a non-zero closed ideal of $A$. 
	We claim that $A/P$ is a finite-dimensional algebra.
	
	For $0\neq x\in P$ we have $0\neq x^*x\in P$ 
	since 
	$\Vert x^*x\Vert=\Vert x\Vert^2$. Consequently, $P^*P\neq 0$. Since  $P^*P\subseteq P^*\cap P$ then  $P^*\cap P$ is a non-zero self-adjoint closed  ideal of $A$.   It suffices to prove that the quotient algebra $A/P^*\cap P$ is finite-dimensional. Therefore,  we can assume that $P$ is a nonzero self-adjoint closed ideal of $A$.
	
	Put $I=H(P,*)$. Then $I$ is a closed ideal of $H(A,*)$, $I\neq 0$ and $H(A,*)/I\cong  H(A/P,*)$. 
	Since $H(A,*)/I$ is a finite-dimensional Jordan algebra, then so is $H(A /P, *)$.  
	By Corollary~\ref{co2}, $H(A /P, *)$ is a Jordan normed algebra whose norm satisfies the equality $\Vert a^2\Vert =\Vert a\Vert^2$.
	Consequently, the Jordan algebra $H(A /P,*)$ is semisimple.  Therefore, it  has the unity $1$ and $1=e_1+\ldots +e_n$ 
	is a sum of central orthogonal idempotents.
	Moreover, $e_i\circ H(A/P, *)$ is a simple algebra.
	By Lemma~\ref{lm6-II}, 
	$e_1,\ldots e_n\in Z(A/P)$.
	By Remark~\ref{rm1}, 
	$1$ is the unity of the algebra $A/P$. Moreover, 
	$$A/P=e_1\cdot A/P\oplus\ldots\oplus e_n\cdot A/P$$ and $e_i\circ H(A/P, *)=H(e_i\cdot A/P,*).$  
	
	Let $Q$ be an ideal of $e_i\cdot A/P.$ 
	Then $H(Q^*Q,*)$ is an ideal of $H(e_i\cdot A/P,*).$ 
	Therefore, either $H(Q^*Q,*)=H(e_i\cdot A/P,*)$ or $H(Q^*Q,*)=0$.  Consequently, either $Q=e_i\cdot A/P$ or $Q=0$,  
	that is, $e_i\cdot A/P$ is simple.
	
	Therefore, $e_i\cdot A/P$ is also finite-dimensional, 
	as previously shown. 
	Hence, $A/P$ is finite-dimensional, as claimed.\qed
	
	We now provide a simpler proof of Theorem \ref{th1} 
	for $C^*$-algebras.
	
	\begin{theor} \label{th1_I}
		Let $A$ be a $C^*$-algebra. 
		The algebra $A$ is just-infinite 
		if and only if
		$H(A,*)$ is just-infinite.\end{theor}
	
	\Proof 
	Assume that $A$ is a just-infinite $C^*$-algebra. 
	Then $H(A,*)$ is infinite dimensional over $\mathbb{R}$ since $A=H(A,*)+iH(A,*)$.
	Let $I$ be a closed non-zero ideal of $H(A,*)$. 
	By Lemma \ref{lm3}, $I=H(P,*)$, 
	where $P$ is a closed ideal of the $C^*$-algebra $A$;
	and there is an isomorphism of algebras
	$H(A,*)/I\cong H(A/P,*)$. 
	Since $H(A/P,*)$ is finite-dimensional 
	then $H(A,*)/I$ is finite-dimensional. 
	Therefore, $H(A,*)$ is a just-infinite Jordan algebra.
	
	Conversely, 
	let $H(A,*)$ be a just-infinite Jordan algebra. 
	Then $A$ is infinite dimensional  by Lemma~\ref{lm5}.
	Let $P$ be a non-zero closed ideal of $A$. 
	Then $P$ is a self-adjoint ideal 
	(see \cite{Dixmier}[Proposition~1.8.2]).
	Hence, $I=H(P,*)$ is a closed ideal of $H(A,*)$, 
	$I\neq 0$ and $H(A,*)/I\cong  H(A/P,*)$. 
	Since $ H(A,*)/I$ is finite-dimensional,
	then $H(A/P,*)$ is a finite dimensional Jordan algebra. Since $A/P=H(A/P,*)+iH(A/P,*)$, 
	then $A/P$ is finite-dimensional. 
	Hence, $A$ is just-infinite.\qed
	
	Let $A$ be a $C^*$-algebra and let $X$ a subset of $A$. 
	By $\langle X \rangle$, we denote the 
	$\mathbb{R}$-subalgebra of $A$ generated by $X$ 
	that is closed under the involution map $*$.
	
	\begin{co} \label{co3} 
		A $C^*$-algebra $A$ is just-infinite 
		if and only if 
		the closure of algebra $\langle H(A,* )\rangle $ 
		is a just-infinite $\mathbb{R}$-algebra.  
	\end{co}
	
	\begin{lm} \label{lm6-I} 
		Let $B$ be a self-adjoint $\mathbb{K}$-subalgebra 
		of a $C^*$-algebra $A$. 
		Assume that the Jordan algebra $H(B,*)$ is prime. 
		Then $B$ is also prime.  
	\end{lm}
	
	\Proof 
	The algebra $B$ contains no nilpotent ideals. 
	Let $P$ and $Q$ be ideals of $B$. 
	Assume $PQ=0$. 
	Then $P^*P$ and $QQ^*$ are ideals of $B$, 
	and $P^*PQQ^*=0$.
	By Lemma~\ref{lm6}, we get $QQ^*P^*P=P^*P\cap QQ^*=0$. Therefore, $H(P^*P,*)\circ H(QQ^*,*)=0$. 
	Since $H(P^*P,*)$ and $H(QQ^*,*)$ 
	are  ideals of $H(B,*$) 
	then, by hypothesis, we have 
	either $H(P^*P,* )=0$ or $H(QQ^*,*)=0$. 
	It implies that $P=0$ or $Q=0$ since $p^*p\in H(P^*P,*)$ and $qq^*\in H(QQ^*, *)$ 
	for all $p\in P$ and $q\in Q$.
	Consequently, $B$ is a prime algebra. \qed
	
	\section{\bf The structure of just-infinite $JB$-algebras}
	\hspace*{\parindent}\vspace*{-10pt}
	
	In this section, 
	we describe the structure of just-infinite Jordan Banach algebras. 
	
	\begin{df}
		Every norm-closed Jordan subalgebra of $B(\mathcal{H})_{sa}$ is called a $JC$-algebra \cite{AlShSt, Han-OlSt}.\end{df}
	
	Let's consider an important example of $JB$-algebra, called the spin factor.
	
	\begin{ex} 
		Let $\mathcal{H}$ be a real Hilbert space 
		of dimension at least $3$ and
		$e$ be a distinguished unit vector in $\mathcal{H}$. 
		Let $V =span(e)^\perp$, so
		$\mathcal{H}=\mathbb{R}e\oplus V$. 
		Then $\mathcal{H}$ becomes a spin factor 
		when equipped with the Jordan product
		$$(\alpha e + a) \circ  (\beta e + b) = (\alpha \beta  + (a, b))e + (\alpha b + \beta a), \alpha,\beta\in \mathbb{R},  a, b \in V.$$
		The  norm is defined by  $$\Vert \alpha e + a\Vert =|\alpha|+\Vert a\Vert.$$ 
		Every  spin factor is a $JC$-algebra \cite{AlShSt}. \end{ex}
	
	We will also show that a just-infinite $JB$-algebra is a $JC$-algebra.

	\begin{theor} \label{th2}
		Let $J$ be a just-infinite $JB$-algebra. 
		Then $J$ is prime.
	\end{theor}
	
	\Proof 
	Let $P$ and $Q$ be closed non-zero ideals of $J$. 
	Assume $PQ=0$. 
	The algebra $J$ does not contain nilpotent elements
	since $\Vert x^2\Vert=\Vert x\Vert^2$ for any $x\in J$. 
	Therefore, by Lemma \ref{lm6}, $P\cap Q=0$. 
	Then the algebra $J$ is embedded into the algebra $J/P\oplus J/Q$. 
	Since $J$ is a just-infinite algebra
	then $J/P\oplus J/Q$ is finite-dimensional. 
	Therefore, $J$ is finite-dimensional.
	This is a contradiction. 
	Consequently, $J$ is prime.\qed
	
	In any Jordan algebra we can define the linear mapping $$U_a:x\mapsto 2(x\circ a)\circ a -x\circ a^2.$$
	
	Every $JB$-algebra $J$ is non-degenerate, that is, $U_a(J)\neq 0$ for any $a\neq 0$.
	
	Let $A$ be an associative algebra. 
	Let $J$ be a subalgebra of the corresponding
	Jordan algebra $A^{(+)}$. Then for any elements $a,b,c\in J$ 
	the following identities hold: 
	\bee \label{Qcomm1}[a,b]^2=4a\circ (a,b,b)-2(a^2, b,b);\eee
	\bee \label{Jass}(a,b,c)=\frac{1}{4}[b,[a,c]],\eee 
	where $(a,b,c)$ 
	denotes the associator of the elements $a,b,c$ in $J$.
	
	Define the mapping $D(a,b):J\mapsto J$ by setting
	\bee \label{diff}  D(a,b)(c)=(a,c,b)=\frac{1}{4}[c,[a,b]]. \eee
	The mapping $D(a,b)$ is a derivation of 
	the Jordan algebra $J$.
	
	Recall that in any Jordan algebra, 
	the following identity holds
	\bee \label{Jordan} (x y, z, t) + (xt, z, y) + (yt, z, x) = 0.\eee 
	
	Let $SJ\langle X\rangle$ be a free special Jordan algebra generated by the set of the free generators $X$. 
	Let $x,y,z,$ and $t$ be elements of the set $X$. 
	Then the element $$\{x,y,z,t\} = xyzt + tzyx$$ 
	is called {\it a tetrad}.
	
	Let us fix the elements 
	$x,y,z \in X $. 
	In~\cite{Zelmanov1983}, 
	the following elements were defined: 
	$c=2D(x,y)^2(z) \circ D^3(x,y) (z)$ and 
	$f(x,y,z)=c^{180}$.
	Let $f_0,\ldots,f_s$ be 
	all possible linearizations of the element $f$, and 
	let $f_i(J)$ denote the value of $f_i$ 
	in the Jordan algebra $J$.
	
	In~\cite{Zelmanov1983}, the ideal 
	$T = \sum_{i=0}^{s} I(f_i)$ of $J$ was defined, 
	where $I(f_i)$ is the ideal of $J$ generated by the set $f_i(J)$.
	
	Let $A$ be an associative algebra with an involution $j$. Let $J$ be a subalgebra of $H(A,j)$. 
	Let $B$ be the associative subalgebra generated by $J$ in $A$. 
	It is well known (see \cite{ZHSlShShir, Han-OlSt}) that the algebra $H(B,j)$ 
	is generated by $J$ and $\{J, J, J,J\}$.
	
	As was shown in~\cite{Zelmanov1983}, we have  
	\bee \label{eatTetrad} T\circ\{J,J, J,J\}\subseteq J\,, \{T,J, J,J\}\subseteq J.\eee
	
	In~\cite{Zelmanov1983} the following was proved 
	
	\begin{theor} \label{th3}
		Let $J$ be a prime non-degenerate Jordan algebra. 
		Then one of the following cases holds:
		1) $J$ is an Albert ring;
		
		2) $J$ is the central order in a Jordan algebra 
		of a non-degenerate bilinear form;
		
		3) $J$ contains a non-zero ideal $I$, 
		which is isomorphic to an algebra $A^ {(+)}$, 
		for a prime associative algebra $A$. 
		The algebra $J$ is embedded in the two-sided
		Martindale quotient ring of the algebra $A$, that is, $A^{(+)}\unlhd J\subseteq Q(A)^{(+)}$;
		
		4) $J$ contains a non-zero ideal $I$, 
		which is isomorphic to an algebra $H(A,j)$, 
		where $A$ is a prime associative algebra
		with an involution $j: A \mapsto A$. 
		The algebra $J$ is embedded 
		in the two-sided Martindale quotient ring, 
		the involution $j$ is extended in the unique way to the involution $j: Q(A)\mapsto Q(A)$ 
		of the quotient ring; and 
		$H(A,j)\unlhd J\subseteq H(Q(A),j)$.
	\end{theor}

	Note that in cases 3 and 4 of the theorem, 
	the ideal $T$ is not equal to $0$ 
	(see \cite{Zelmanov1983}).
	
	Using Theorem~\ref{th3}, we prove
	
	\begin{theor} \label{th4}
		Let $J$ be a just-infinite $JB$-algebra. 
		Then $J$ is a $JC$-algebra. 
		Moreover, one of the following cases holds:
		
		$(i)$ $J$ is an infinite-dimensional spin factor;
		
		$(ii)$ there is a $C^*$-algebra $A$ 
		and just-infinite norm-closed real $*$-subalgebras $A_1$ and $A_2$ of $A$ 
		such that $$H(A_1,*)\unlhd J\subseteq H(A_2,*).$$
		Moreover, 
		the $C^*$-algebra $A$ is the closure of $A_2+iA_2$, and $A_1$ is an ideal of $A_2$.
	\end{theor}
	
	First we will prove several lemmas.
	
	\begin{lm} \label{lm7} 
		Let $J$ be a $JC$-algebra. 
		Assume that $J$ contains a nonzero ideal $I$, 
		which is isomorphic to an algebra $A^ {(+)}$, 
		where $A$ is an associative algebra.
		Then $I^2\in Z(J).$    \end{lm}
	
	\Proof
	Let $\phi: A^{(+)}\mapsto I$ 
	be an isomorphism of Jordan algebras.  
	Let $a,b\in  A$.
	Then, by (\ref{Qcomm1}), 
	$$\phi([a,b])^2=\phi([a,b]^ 2)=\phi(4a\circ
	(a,b,b)-2(a^2, b,b))=$$
	$$4\phi(a)\circ (\phi(a),\phi(b),\phi(b))-2(\phi(a)^2, \phi(b),\phi(b)).$$
	The element $\phi([a, b])\in I\subseteq J$. 
	Therefore, $\phi([a,b])^2$ is a positive element in $J$.
	
	The algebra $J$ is a $JC$-algebra, 
	i.e. $J$ is a norm-closed subalgebra of $B(\mathcal{H})_{sa}$ 
	for some complex Hilbert space $\mathcal{H}$. 
	Consider the commutator
	$[\phi(a),\phi(b)]$ of elements $\phi(a), \phi(b)$ 
	in the algebra $B(\mathcal{H})$. 
	Then
	$$[\phi(a),\phi(b)]^2= 4\phi(a)\circ (\phi(a),\phi(b),\phi(b))-2(\phi(a)^2, \phi(b),\phi(b)).$$
	The element $[\phi(a),\phi(b)]$ 
	is skew-symmetric in $B(\mathcal{H})$. 
	Therefore, $i[\phi(a),\phi(b)]$ is a self-adjoint element in $B(\mathcal{H})$. 
	Then $(i[\phi(a),\phi(b)])^2$ 
	is a positive element in $B(\mathcal{H})$. 
	Furthermore,
	$$\phi([a,b])^2+(i[\phi(a),\phi(b)])^2=0.$$
	
	Consequently, for all elements $x,y\in I$ 
	we get  $[x,y]=0$, where
	$[x,y]$ is the commutator in $B(\mathcal{H})$.
	
	Let $x,y\in I, z\in J$.
	By (\ref{Tcomm}), we get
	$$[x\circ y,z ]=-[y\circ z,x ]-[z\circ x,y ]=0,$$ 
	since $y\circ z,z\circ x\in I$.
	By (\ref{Jass}), 
	$$(x\circ y,u,z)=\frac{1}{4}[u,[x\circ y, z]]=0$$ 
	for all $u,z\in J$. 
	Therefore, $I^2\subseteq Z(J)$.
	\qed
	
	In the following lemmas we assume that 
	$J$ is a $JC$-algebra, and 
	$A$ is the norm-closed associative $\mathbb{R}$-subalgebra of $B(\mathcal{H})$, generated by $J$. 
	Let $\overline{T}$ be the closure of the ideal $T$, and let $B$ be the closure of the algebra $A+iA$. 
	Then $B$ is a $C^*$-algebra.
	
	\begin{lm} \label{lm8}  
		The vector space $\overline{T}$ is an ideal of $H(A,*)$.  \end{lm}
	
	\Proof 
	Since the involution $*$ is a continuous map we have $A^*\subseteq A$. 
	It is clear that 
	$\overline{T}\circ\{J,J,J,J\}\subseteq J\,, \{\overline{T},J,J,J\}\subseteq J$ 
	since $T\circ\{J,J,J,J\}\subseteq J$, 
	$\{T,J,J,J\}\subseteq J$ 
	and the multiplication operation is continuous on $B(\mathcal{H})$.
	
	We note that for any closed ideal $I$ of $J$ 
	the equality $I\circ I=I$ holds. 
	Indeed, it is well known that for any $x\in I$ 
	we have $x=x_+-x_-$, where 
	$x_+$ and $x_-$ are the positive elements of $I$. 
	Then $x_+=u^2$ and $x_-=w^2$, where $u,w$ are the elements of the closed subalgebra generated by $x$. 
	Consequently, $I=I\circ I$.
	
	We claim that $\overline{T}$ is an ideal of $H(A,*)$. 
	Let $A_0$ be the subalgebra of $A$ generated by $J$. 
	It is sufficient to prove that 
	$\overline{T}\circ H(A_0,*)\subseteq \overline{T}.$
	
	Let $V=J+\{J,J,J,J\}$. 
	Then $\overline{T}\circ V\subseteq J$, 
	because $T\circ V\subseteq J$ and 
	the multiplication operation $(\circ)$ is continuous. 
	We have 
	$\overline{T}=\overline{T}^2=(\overline{T}\circ  \overline{T})\circ  \overline{T}=\overline{T}^3$. Therefore, by (\ref{Jordan}), 
	$$(\overline{T},  J,V)\subseteq (\overline{T}\circ \overline{T},  J,V)\subseteq (\overline{T}\circ V  ,  J,\overline{T}) \subseteq (J  ,   J,\overline{T})\subseteq \overline{T}. $$ 
	Consequently,
	$$\overline{T}\circ V\subseteq \overline{T}^3\circ V\subseteq \overline{T}^2\circ(\overline{T} \circ   V) +(\overline{T}^2,   \overline{T},V)\subseteq  \overline{T}\circ J+\overline{T}\subseteq  \overline{T} .$$
	
	Let $x, y\in H(A_0,*)$ and 
	$\overline{T}\circ x,\overline{T}\circ y \in \overline{T}$. 
	Then 
	$$\overline{T}\circ (x\circ y)\subseteq (\overline{T}\circ x)\circ y+(\overline{T},x,y)\subseteq  \overline{T}+(\overline{T}\circ y ,x,\overline{T})\subseteq  \overline{T}+(\overline{T} ,x,\overline{T})\subseteq \overline{T}.$$
	Since $H(A_0,*)$ is generated by $V$, we have $\overline{T}\circ H(A_0,*)\subseteq \overline{T},$ i.e. $\overline{T}$ is an ideal of the algebra $H(A, *)$.\qed
	
	\begin{lm} \label{lm9} 
		Let $I$ be a nonzero ideal of $A$. 
		The following statements hold:
		
		(i)  if  $J\cap I=0$ then $\overline{T}I=I\overline{T}=0$,
		
		(ii) if $J$ is a prime algebra and $T\neq 0$, then $J\cap I  = 0$ if and only if 
		$\overline{T}I = I\overline{T}=0$,
		
		(iii) if $A$ is a prime algebra and $T\neq 0$ 
		then $J\cap I\neq 0$ 
		for any nonzero ideal $I$ of $A$. \end{lm}
	
	\Proof 
	Let $J\cap I=0$. 
	Recall  that   $H(I,*)$ is a nonzero ideal of $H(A, *)$. 
	By Lemma \ref{lm8}, 
	$$\overline{T}\circ H(I,*)\subseteq \overline{T}\cap H(I,*)\subseteq J\cap I=0.$$ 
	Since $\overline{T}=\overline{T}\circ \overline{T}$ 
	then by (\ref{Tcomm}), 
	$$[\overline{T},H(I,*)]\subseteq [\overline{T}\circ \overline{T},H(I,*)]\subseteq[\overline{T}\circ H(I,*),\overline{T}]=0.$$
	Consequently, $\overline{T}H(I,*)=H(I,*)\overline{T}=0$.
	
	Let $x\in I$. 
	Since $x^*x\in H(A,* )\cap I\subseteq H(I,*)$, then $tx^*xt=0$ for any $t\in\overline{T}$. Consequently, $(xt)^*xt=tx^*xt=0$.
	From here, we obtain
	$\Vert xt\Vert^2=\Vert(xt)^*xt\Vert=0$. 
	Therefore, $xt=0$. 
	Similarly, $tx=0$.
	
	Thus, $I\overline{T}=\overline{T}I=0$, 
	that is, we proved $(i)$.
	
	Prove $(ii)$. 
	Let $J$ be a prime algebra and let $\overline{T}\neq 0$. Assume $\overline{T}I = I\overline{T}=0$. 
	The vector space $J\cap I$ is an ideal of $J$. 
	Moreover, $\overline{T}\circ (J\cap I)=0$. 
	It follows that $\overline{T}=0$ or $J\cap I=0$ since 
	$J$ is a prime algebra. 
	By hypothesis $\overline{T}\neq 0$. 
	It means that $J\cap I=0$.
	
	Prove $(iii)$. 
	Let $A$ be a prime algebra and let $\overline{T}\neq 0$. Assume $J\cap I=0$ for some ideal $I$ of $A$. 
	Then $\overline{T}I=0$ by $(i)$. 
	Therefore, 
	$$(A^\#\overline{T}A^\#)(A^\#IA^\#)\subseteq A^\#\overline{T}IA^\#=0.$$
	Consequently, $I=0$.
	\qed
	
	\begin{lm} \label{lm10} 
		Let $J$ be a prime algebra, and let $T\neq 0$. 
		Then either $A$ is prime or
		$A$ contains the largest ideal $K$ of $A$ 
		such that $J\cap K=0$.  
	\end{lm}
	
	\Proof 
	Assume that $A$ is not a prime algebra. 
	Then there exist ideals $P$ and $Q$ of $A$ 
	such that $PQ = 0$. 
	By Lemma~\ref{lm6}, we get $QP=0$.
	Since the algebra $J$ is prime, 
	either $J \cap P = 0$ or $J \cap Q = 0$. 
	We assume that $J \cap P = 0$.
	By Lemma~\ref{lm9}, we get 
	$\overline{T}P = P\overline{T} = 0$. 
	Consequently, there exists a largest ideal $K$ of $A$ 
	such that $\overline{T}K = K\overline{T} = 0$. 
	Since $J$ is a prime algebra and $\overline{T} \neq 0$, then $J \cap K = 0$ by Lemma \ref{lm9}.
	
	We claim that $K$ is the largest ideal of $A$, 
	with the condition $J \cap K = 0$. 
	Let $I$ be an ideal of $A$ such that $J\cap I=0$. 
	By Lemma~\ref{lm9}, we have $\overline{T}I=I\overline{T}=0$. 
	Therefore, $I\subseteq K$.\qed
	
	\begin{lm} \label{lm11} 
		Let $J$ be a prime algebra, and assume $T\neq 0$. 
		Then there exists a $C^*$-algebra $B(J)$ 
		such that $J$ is a subalgebra of $H(B(J),*)$. 
		A norm-closed $\mathbb{R}$-subalgebra $A(J)$ of $B(J)$, generated by $J$, is prime.
		Moreover, $B(J)$ is the closure of $A(J)+iA(J)$. 
	\end{lm}
	
	\Proof 
	If $A$ is a prime algebra then everything is proven. Namely, $A(J)=A$, $B(J)=B$.
	
	Therefore, by Lemma~\ref{lm10} and Lemma~\ref{lm9},  we can assume that
	$A$ contains the largest ideal $K$ of $A$ such that $\overline{T}K=K\overline{T}=0$. 
	Then $K+iK$ is an ideal of $C^*$-algebra $B$ and $\overline{T}(K+iK)=(K+iK)\overline{T}=0$.
	
	Consequently, the $C^*$-algebra $B$ 
	contains the largest ideal $K_1$ such that
	$\overline{T}K_1=K_1\overline{T}=0$. 
	Since $J$ is a prime algebra we have $J\cap K_1=0$. 
	It is clear that 
	$K_1$ is a closed self-adjoint ideal of $B$. 
	Then the quotient algebra $B/K_1$ is a $C^*$ algebra.
	
	Since $J\cap K_1=0$ we can assume that $J$ is 
	a $JB$-subalgebra of the Jordan algebra $H(B/K_1,*)$.
	
	Let $A(J)$ be a norm-closed $\mathbb{R}$-subalgebra of $B/K_1$ generated by $J$. Denote by  $\tilde{A}$  the image of $A$  under the canonical homomorphism $B \mapsto B/K_1$.  The canonical homomorphism $B\mapsto B/K_1$ 
	is continuous. Therefore,  $\tilde{A}\subseteq A(J)$. The algebra $B(J)$ is the closure of $A(J)+iA(J)$ 
	in the $C^*$-algebra $B/K_1$.

	We claim that $J\cap \tilde{I}\neq 0$ for any nonzero ideal $\tilde{I}$ of $A(J)$. Assume $J \cap \tilde{I} = 0$ 
	for some nonzero ideal $\tilde{I}$ of $A(J)$. By Lemma \ref{lm9}, we have  $\overline{T}\tilde{I}=\tilde{I}\overline{T}=0$. Since $\tilde{A} \subseteq A(J)$ then $\tilde{A}\tilde{I} \subseteq\tilde{I}$ and $\tilde{I}\tilde{A} \subseteq\tilde{I}$.  
	Let $I$ be the preimage of the ideal $\tilde{I}$ 
	in the algebra~$B$. Then $\overline{T}I\subseteq K_1$. Because $\overline{T}=\overline{T}\circ\overline{T}$, and $\overline{T}K_1=0$, we get
	$$\overline{T}I\subseteq (\overline{T}\circ \overline{T})I\subseteq \overline{T}K_1=0.$$
	Similarly, $I\overline{T}=0$. Moreover, $AI\subset I$ and $IA\subset I$. Let $K$ be the closure of $I+iI$ in the algebra $B$. Then $K$ is an ideal of $B$ and  $K\overline{T}=\overline{T}K=0$. 
	The ideal $K_1$ is the largest ideal of $B$ for which $\overline{T}K_1=K_1\overline{T}=0$. Therefore, $K\subset K_1$. 
	Consequently, $\tilde{I}=0$, a contradiction.
	
	Thus, $J\cap \tilde{I}\neq 0$ for any nonzero ideal $\tilde{I}$ of $A(J)$.
	By Lemma~\ref{lm10}, we obtain that $A(J)$ is prime.
	\qed

	We further use the notation $A(J)$ defined in Lemma~\ref{lm11}.
    
	Recall that if a Jordan algebra $J$ is a finite-dimensional subalgebra of $A^{(+)}$, then the subalgebra of $A$ generated by $J$ is also finite-dimensional.
	
	\begin{lm} \label{lm12}
		Let $J$ be a just-infinite algebra, and assume $T \neq 0$. Then the algebra $A(J)$ is just-infinite. 
	\end{lm}
	
	\Proof 
	By Theorem~\ref{th1}, it suffices to prove that 
	the algebra $H(A(J),*)$ is just-infinite. 
	Let $I$ be a norm-closed ideal of $H(A(J),*)$. 
	Then by Lemma~\ref{lm3}, $I=H(P,*)$, 
	where $P$ is a norm-closed ideal of $A(J)$. 
	We also have $H(A(J),*) /I\cong H(A/P,*)$.
	
	Let $\langle J\rangle $ be a $\mathbb{R}$-subalgebra 
	of $B(J)$, generated by $J$.
	By Lemma~\ref{lm11}, 
	$A(J)$ is the closure of $\langle J\rangle$.
	Then it is clear that $H(A(J),*)$ is the closure of $H(\langle J\rangle,*)$.
	Moreover, $A(J)$ is prime. 
	Then by Lemma~\ref{lm9}, $J\cap P\neq 0$.
	Therefore, $J\cap I=J\cap P\neq 0$.
	
	Consider the quotient algebra $A(J)/P$. 
	Then the subalgebra $\langle J/P\rangle$ of $A(J)/P$ 
	is generated by $J/P$.
	Since $J/P \cong J/(J \cap P)$ and $J$ is just-infinite, we obtain that $J/P$ is finite-dimensional.
	Therefore, $\langle J/P \rangle $ is  a finite-dimensional algebra as well.
	Consequently, $H(\langle J\rangle,*)/P$ 
	is finite-dimensional, since
	$H(\langle J\rangle,*)/P\cong H(\langle J/P \rangle ,*)$.
	We have 
	$$H(\langle J\rangle,*)/I\cong H(\langle J\rangle,*)/ H(\langle J\rangle,*)\cap I=
	H(\langle J\rangle,*)/H(\langle J\rangle,*)\cap P\cong H(\langle J\rangle,*)/P.$$ 
	Hence, $H(\langle J\rangle,*)/I$ is finite-dimensional.
	
	The canonical homomorphism 
	$\pi :H(A(J),*)\mapsto H(A(J),*)/I$ is continuous.  Consider an element $\tilde{h}\in\pi(H(A(J),*))$. 
	Then $\tilde{h}=\pi(h)$ for some $h\in H(A(J),*)$.
	Since $H(A(J),*)$ is the closure of 
	$H(\langle J\rangle,*)$ 
	then there exist a sequence $\{h_n\}$ 
	of the elements of $H(\langle J\rangle,*)$ 
	such that $h_n\rightarrow h$.
	Consequently, $\pi(h_n)\rightarrow \pi(h)$ 
	since $\pi$ is continuous. 
	As $\pi(h_n)\in H(\langle J\rangle,*)/I$  and 
	$H(\langle J\rangle,*)/I$ is finite-dimensional, $H(\langle J\rangle,*)/I$ is norm-closed. 
	Therefore, $\tilde{h}=\pi(h)\in H(\langle J\rangle,*)/I$, that is, $H(A(J),*)/I=H(\langle J\rangle,*)/I$. 
	It means that $H(A(J),*)/I$ is finite-dimensional.
	
	Thus, we proved that $H(A(J),*)$ is just-infinite.\qed
	
	{\bf Proof of Theorem~\ref{th4}} 
	Being a $JB$-algebra 
	$J$ is a non-degenerate Jordan algebra.
	By Theorem~\ref{th2}, 
	$J$ is prime, since $J$ is a just-infinite algebra. By Theorem \ref{th3}, $J$ is the algebra of cases 1)-4).
	
	The center $Z(J)$ of the $JB$-algebra $J$ 
	is an associative commutative Banach $\mathbb{R}$-algebra.
	Therefore, if  $Z(J) \neq 0$, then $Z(J) = \mathbb{R}$. Otherwise, $Z(J)$ would contain zero divisors, which is not possible since $J$ is a prime algebra.
	
	If $J$ is the Albert ring, then 
	$Z(J) = \mathbb{R}$ and $J$ is a $27$-dimensional algebra over $\mathbb{R}$. It is not possible since 
	$J$ is an infinite-dimensional algebra over $\mathbb{R}$.
	
	If $J$ is a special Jordan algebra 
	then $J$ is $JC$-algebra (see \cite{AlShSt,  Han-OlSt}), i.e.
	$J$ is a norm-closed subalgebra of $B(\mathcal{H})_{sa}$
	for some complex Hilbert space $\mathcal{H}$.
	
	Assume $J$ is a central order in a Jordan algebra 
	with a non-degenerate bilinear form, 
	then the center $Z(J)$ is  the real numbers $\mathbb{R}$ again, and 
	$J$ itself is an infinite-dimensional Jordan algebra 
	with a non-degenerate bilinear form. 
	Since $J$ is a $JB$-algebra, then $J$ is formally real, i.e.
	$\sum_{k=1}^na_k^2=0$ implies $a_k=0$ for $k=1,\ldots,n$. Therefore,
	the bilinear form on $J$ is positive definite. 
	Then $J$ is a $JB$-factor of type $I_2$ (see \cite{AlShSt}).
	Consequently, $J$ is an infinite-dimensional spin factor (see \cite{AlShSt}).
	
	Assume that $J$ contains a non-zero ideal $I$,
	which is isomorphic to an algebra $A^ {(+)}$, 
	where $A$ is an associative algebra. 
	By Lemma~\ref{lm7}, we have $I^2\subseteq Z(J)$. 
	Then $Z(J)\neq 0$ since $I^2\neq 0$. 
	It implies that $Z(J)=\mathbb{R}$. 
	Consequently, $J=I=\mathbb{R}$. 
	Therefore, $\dim_\mathbb{R} J = 1$. 
	This is a contradiction because $J$ has infinite dimension over $\mathbb{R}$.
	
	Assume that $J$ contains a non-zero ideal $I$, 
	which is isomorphic to an algebra $H(A_0,j)$, 
	where $A_0$ is an associative algebra with
	an involution $j$. 
	In this case, the ideal $T$ is non-zero. 
	By Lemma~ \ref{lm11}, there exists a $C^*$-algebra $B(J)$ and a  $\mathbb{R}$-subalgebra $A(J)$ of $B(J)$ 
	that is the closure of the
	$\mathbb{R}$-subalgebra $\langle J\rangle$, generated by $J$ such that $$J\subseteq H(A(J),*)\subseteq H(B(J),*).$$
	Then $H(A(J),*)$ is the closure of 
	$H(\langle J\rangle ,*)$. 
	By repeating the proof of Lemma~\ref{lm8}, 
	we see that the closure $\overline{T}$ of $T$ 
	is an ideal in $H(A(J),*)$. 
	By Lemma~\ref{lm12}, $A(J)$ is just-infinite.
	
	Put $A=B(J)$ and $A_2=A(J)$.  By Lemma \ref{lm11}, $A$ is the closure of $A_2+iA_2$. 
	By Lemma~\ref{lm3}, $\overline{T}=H(A_1,*)$,
	where $A_1$ is a self-adjoint norm-closed ideal of $A_2$.
	
	We claim that any closed ideal of $\overline{T}$ 
	is also an ideal of $J$. Indeed, 
	assume that $P$ is a closed ideal of $\overline{T}$. 
	As $P$ is a $JB$-algebra, for any $h\in P$ we have 
	$h=h_+-h_-$, 
	where $h_+$ and $h_-$ are the positive elements of $P$. 
	Then, we can extract the root of any degree from 
	$h_+$ and $h_-$ in the $JB$-algebra $P$. 
	Therefore, $P^3=P$. 
	Consequently, by (\ref{Jord-Id}), $P$ is an ideal of $J$. 
	This implies that $\overline{T}$ is just-infinite. Therefore, by Theorem \ref{th1}, the algebra $A_1$ is just-infinite.
	Theorem~\ref{th4} is proven.
	\qed
	
	\begin{center}
		{\bf  Acknowledgments}\end{center}
	The research was carried out within the framework of the Sobolev
	Institute of Mathematics state
	contract (project FWNF-2026-0017).

	\Addresses

\begin{thebibliography}{99}
		
		\bibitem{FS} 
		D. R. Farkas and L. W. Small, 
		“Algebras which are nearly finite dimensional 
		and their identities”, 
		Isr. J. Math., 127, 245--251 (2002).
		
		\bibitem{PT} 
		D. S. Passman and W. V. Temple, 
		“Representations of the Gupta–Sidki group”, 
		Proc. Am. Math. Sos., 124:5, 1403--1410 (1996).
		
		\bibitem{RRZh} 
		Z. Reichstein, D. Rogalski, J. J. Zhang, 
		“Projectively simple rings”, 
		Adv. Math., 203, 365--407 (2006).
		
		\bibitem{FarPen} 
		J.Farina and C. Pendergrass-Rice, 
		“A few properties of just infinite algebras”,
		Commun. Algebra, 35, 1703--1707 (2007).
		
		\bibitem{BellFarPen} 
		J.Bell, J. Farina, C. Pendergrass-Rice, 
		“Stably just infinite rings”,
		J. Algebra, 319, 2533--2544 (2008).
		
		\bibitem{ZhP} 
		V. N. Zhelyabin, A. S. Panasenko, 
		“Nearly finite-dimensional Jordan algebras”, 
		Algebra and Logic, 57:5, 336--352 (2018). 
		
		\bibitem{ZhPan} 
		V.N. Zhelyabin, A.S. Panasenko, 
		“Herstein's construction for just infinite superalgebras”, 
		Sib. electronic math. reports, 14, 1317--1323 (2017).
		
		\bibitem{ShalZel} 
		A. Shalev and E. I. Zelmanov, 
		“Narrow Lie algebras: A coclass theory and a characterization of the Witt algebra”, 
		J. Algebra, 189, 294--331 (1997).
		
		\bibitem{GMS} 
		N. Gavioli, V. Monti and C. M. Scoppola, 
		“Just infinite periodic Lie algebras”.
		In {\it Finite Groups 2003}, 
		Walter de Gruyter, Berlin. 73--85 (2004).
		
		\bibitem{Wilson} 
		J. S. Wilson, 
		“Groups with every proper quotient finite”, 
		Proc. Cambridge Philos. Soc., 69:3, 373--391 (1971).
		
		\bibitem{GupSid} N. Gupta, S. Sidki, 
		“Some infinite p-groups”,
		Algebra i Logika, 22:5, 584--589 (1983).
		
		\bibitem{GrigShum} R. Grigorchuk and P. Shumyatsky, 
		“On just-infinite periodic locally soluble groups”, 
		Arch. Math., 109, 19--27 (2017).
		
		\bibitem{GMR} R. Grigorchuk, M. Musat and M. R${\o}$rdam,
		“Just-infinite $C^*$-algebras”,
		Comment. Math. Helv., 93, 157--201 (2018).
		
		\bibitem{AlShSt} 
		Erik M. Alfesn, Frederic W. Shultz and Erling  St{\o}rmer,
		“A Gelfand-Neumark theorem for Jordan algebras”,
		Adv. Math., 28, 11--56 (1978).
		
		\bibitem{EffSt} 
		E. Effros and E. St{\o}rmer, 
		“Jordan algebras of self-adjoint operators”, 
		Trans. Amer. Math. Soc., 127, 313--316 (1967).
		
		\bibitem{Stor} 
		E. St{\o}rmer, 
		“On the Jordan structure of $C^*$-algebras”, 
		Trans. Amer. Math. Soc., 120, 438--447 (1965).
		
		\bibitem{Stor-I} 
		E. St{\o}rmer, 
		“Jordan algebras of type I”,
		Acta Math., 115, 165--184 (1966).
		
		\bibitem{Stor-II} 
		E. St{\o}rmer, 
		“Irreducible Jordan algebras of self-adjoint operators”, Trans. Amer. Math. Soc., 130, 153--166 (1968).
		
		\bibitem{Top} 
		D. Topping, 
		“Jordan algebras of self-adjoint operators”, 
		Mem. Amer. Math. Soc., 53 (1965).
		
		\bibitem{Top-I} 
		D. Topping, 
		“An isomorphism invariant for spin factors”, 
		J. Math. Mech., 15, 1055--1064 (1966).
		
		\bibitem{Dixmier} 
		J. Dixmier, 
		“$C^*$-Algebras”, North-Holland, (1977).
		
		\bibitem{KMcC} 
		Kevin McCrimmon, 
		“On Herstein's theorems relating Jordan and 
		associative algebras”, 
		J. Algebra, 13, 382--392 (1969).
		
		\bibitem{ZHSlShShir}  K. A. Zhevlakov, A. M. Slin’ko, I. P. Shestakov and A. I. Shirshov, 
		"Rings that are nearly associative" [in Russian], Nauka, Moscow (1978).
		
		\bibitem{Zelmanov1983} 
		E.I. Zel'manov, 
		“Prime Jordan algebras. II”, 
		Siberian Math. J., 24:1, 73--85 (1983).
		
		\bibitem{Han-OlSt} 
		H. Hanche-Olsen and E. St{\o}rmer, 
		“Jordan operator algebras”, 
		Monographs and Studies in Mathematics, 21. 
		Pitman (Advanced Publishing Program), Boston, MA, 1984.
		
		\bibitem{Amitsur} 
		S.A. Amitsur,  
		“Rings with involution”, 
		Isr. J. Math., 6, 99--106 (1968).
	\end{thebibliography}
\end{document}